\documentclass[final]{siamltex}

\usepackage{epsfig}
\usepackage{graphicx}
\usepackage{makeidx}
\usepackage{multicol}

\usepackage{crop}
\crop
\makeindex

\usepackage{pdfsync}

\usepackage{amsmath} 
\usepackage{amssymb}  
\newtheorem{Satz}{Theorem} 
\newtheorem{Lemm}{Lemma} 
\newtheorem{Def}{Definition}
\newtheorem{remark}[Lemm]{Remark}

\newcommand{\Simp}{\ensuremath{\mathcal{S}}}

\newcommand \eps   {\varepsilon}

\newcommand \N   {\mathbb{N}}
\newcommand \R   {\mathbb{R}}

\newcommand \K   {\mathcal{K}}
\newcommand \Kinf{\mathcal{K_\infty}}
\newcommand \KL  {\mathcal{KL}}
\newcommand \LL  {\mathcal{L}}

\newcommand \PD   {\mathcal{P}}

\newcommand \SSet   {\mathcal{S}}

\title{Dwell-time conditions for robust stability of impulsive systems  
\thanks{This research is funded by the German Research Foundation (DFG) as a part of Collaborative Research Centre 637 "Autonomous Cooperating Logistic Processes - A Paradigm Shift and its Limitations".}
}

\author{Sergey Dashkovskiy 
\thanks{S. Dashkovskiy is with 
Erfurt University of Applied Sciences, Altonaer Stra\ss e 25, 99085 Erfurt, Germany
({\tt\small sergey.dashkovskiy@fh-erfurt.de})}%
\and Andrii Mironchenko
\thanks{A. Mironchenko is with Department of Mathematics and Computer Science, University of Bremen, Bibliothekstra\ss e 1, 28359 Bremen, Germany
(corresponding author, {\tt\small andmir@math.uni-bremen.de})}%
}

\begin{document}

\maketitle

\begin{abstract}
We prove that impulsive systems, which possess an ISS Lyapunov function, are ISS for impulse time sequences, which satisfy the fixed dwell-time condition. 

If the ISS Lyapunov function is the exponential one, we provide stronger result, which guarantees uniform ISS of the whole system over sequences of impulse times, which satisfy the generalized average dwell-time condition.
\end{abstract}



\pagestyle{myheadings}
\thispagestyle{plain}

\section{Introduction}

In the modeling of the real phenomena often one has to consider systems, which exhibit both continuous and discontinuous behavior.
The general framework for modeling of such phenomena is a hybrid systems theory 
\cite{HCN06}. 
Impulsive systems are hybrid systems, in which the jumps occur only at certain moments of time, which do not depend on the state of the system.
The first monograph devoted entirely to impulsive systems is \cite{SaP95}. Recent developments in this field one can find, in particular, in \cite{HCN06}.

We are interested in stability of the systems with respect to external inputs.  The central concept in this theory is a notion of input-to-state stability (ISS), introduced by E. Sontag in \cite{Son89}, for a survey on related results see \cite{Son08}, \cite{DES11}.

Input-to-state stability of impulsive systems has been investigated in recent papers
\cite{HLT08} (finite-dimensional systems) and \cite{ChZ09} (time-delay systems). 
The stability of interconnections of impulsive systems has been studied in \cite{DKM11}.
The most interesting case is when either continuous or discrete dynamics destabilizes the system. In this case in order to achieve ISS of the system one has to impose restrictions on the density of impulse times, which are called dwell-time conditions.

In \cite{HLT08} it was proved that impulsive systems, which possess an exponential ISS-Lyapunov function are uniformly ISS over impulse time sequences, which satisfy so-called average dwell-time (ADT) condition.
In \cite{ChZ09} a sufficient condition in terms of Lyapunov-Razumikhin functions is provided, which ensures the uniform ISS of impulsive time-delay systems over impulse time sequences, which satisfy fixed dwell-time (FDT) condition (for the difference between ADT and FDT condition see \cite{HeM99b}).

In both papers only exponential ISS-Lyapunov functions have been considered. This imposes serious restrictions on the class of systems, to which the results can be applied. 
Another restrictions arise in the study of interconnections of ISS impulsive systems via small-gain theorems, which are an important tool for construction of ISS-Lyapunov functions for the interconnected systems \cite{JMW96}, \cite{DRW10}. Even if the ISS-Lyapunov functions for all subsystems are exponential, the ISS Lyapunov function of the interconnection may be non-exponential, if the gains are nonlinear. Hence for the most cases there is no tool to check ISS of an interconnection of impulsive systems.

In this paper we prove, that existence of an ISS Lyapunov function (\textit{not necessarily exponential}) for an impulsive system implies the input-to-state stability of the system over the impulsive sequences satisfying \textit{nonlinear FDT} condition. 

For the case, when the impulsive system possesses an exponential ISS-Lyapunov function, we
generalize the result from \cite{HLT08}, by introducing the \textit{generalized average dwell-time (gADT) condition} and proving, that the impulsive system, which possesses an exponential ISS Lyapunov function is uniformly ISS over impulse time sequences, which satisfy the gADT condition. We argue, that generalized ADT condition provides in certain sense tight estimates of the class of impulsive time sequences, for which the system is ISS.

The structure of the paper is as follows.
In Section \ref{Prelim} we provide notation and main definitions.
In Section \ref{ISS_Single_Sys} the sufficient conditions for ISS of impulsive systems via ISS Lyapunov functions are proved.
Section \ref{Schluss} concludes the paper.

\section{Preliminaries}
\label{Prelim}

Let $T = \{t_1,\ t_2,\ t_3, \ldots \}$ be a monotonically increasing sequence of impulse times without finite accumulation points.
Consider a system of the form
\begin{equation}
\label{ImpSystem}
\left \{
\begin {array} {l}
{ \dot{x}(t)=f(x(t),u(t)),\quad t \in [t_0,\infty)  \backslash T,} \\
{ x(t)=g(x^-(t),u^-(t)),\quad t \in T,}
\end {array}
\right.
\end {equation}
where
$f,g: \R^n \times \R^m \to \R^n$.

The first equation of \eqref{ImpSystem} describes the continuous dynamics of the system, and the second describes the jumps of a state at the impulse times.

We assume that $u \in L_{\infty}([t_0,\infty),\R^m)$ and that all inputs possess left limits of $u$ at all times and that they are right-continuous: for all $t \geq t_0$ it holds $u^-(t) = \lim\limits_{s \to t-0}u(s)$.
We endow this space with the supremum norm, which we denote by $\|\cdot\|_{\infty}$.

We assume that $f$ is locally Lipschitz w.r.t. the first argument in order to guarantee existence and uniqueness of solutions of the problem \eqref{ImpSystem}.

Note that from the assumptions on the inputs $u$ it follows that $x:[t_0,\infty) \to \R^n$ is absolutely continuous between the impulses, and $x^-(t) = \lim\limits_{s \to t-0}x(s)$ exists for all $t \geq t_0$.

Equations \eqref{ImpSystem} together with the sequence of impulse times $T$ define an impulsive system.

We need the following classes of functions
\begin{equation}\notag
\begin{array}{ll}
{\PD} &:= \left\{\gamma:\R_+\rightarrow\R_+\left|\ \gamma\mbox{ is continuous, } 
\right. \gamma(0)=0\ \mbox{and }  \gamma(r)>0 \mbox{ for } r>0 \right\} \\
{\K} &:= \left\{\gamma \in \PD \left|\ \gamma \mbox{ is strictly increasing}  \right. \right\}\\
{\K_{\infty}}&:=\left\{\gamma\in\K\left|\ \gamma\mbox{ is unbounded}\right.\right\}\\
{\LL}&:=\left\{\gamma:\R_+\rightarrow\R_+\left|\ \gamma\mbox{ is continuous and strictly}\right. \text{decreasing with } \lim\limits_{t\rightarrow\infty}\gamma(t)=0\right\}\\
{\KL} &:= \left\{\beta:\R_+\times\R_+\rightarrow\R_+\left|\ \right. \beta(\cdot,t)\in{\K},\ \forall t \geq 0,\  \beta(r,\cdot)\in {\LL},\ \forall r >0\right\}
\end{array}
\end{equation}
Denote the Euclidean norm in spaces $\R^k$ by $|\cdot|$ and $\N:=\{1,2,3,\ldots\}$.

We are interested in stability of the system \eqref{ImpSystem} w.r.t. external inputs. To this end we use the following notion:
\begin{Def}
\label{ISS_One_System}
For a given sequence $T$ of impulse times we call system \eqref{ImpSystem} \textit{ input-to-state stable (ISS)} if there exist $\beta\in\KL,\ \gamma\in\K_{\infty}$, such that for all initial conditions $x_0$, for all inputs $u$, $\forall t\geq t_0$ it holds
\begin{equation}
\label{ISS_ImpSys}
|x(t)| \leq \max\{\beta(|x_0|,t-t_0),\gamma(\|u\|_{\infty})\}.
\end{equation}
The impulsive system \eqref{ImpSystem} is \textit{uniformly ISS} over a given set $\Simp$ of admissible sequences of impulse times if it is ISS for every sequence in $\Simp$, with $\beta$ and $\gamma$ independent of the choice of the sequence from the class $\Simp$.
\end{Def}

In the next section we are going to find the sufficient conditions for an impulsive system \eqref{ImpSystem} to be ISS.

\section{ISS of an impulsive system}
\label{ISS_Single_Sys}

For analysis of ISS of impulsive systems we exploit ISS-Lyapunov functions.
\begin{Def}
\label{ISS_Imp_LF}
A smooth function $V:\R^n \to \R_+$ is called an \textit{ISS-Lyapunov function} for \eqref{ImpSystem} if $\exists \ \psi_1,\psi_2 \in \Kinf$, such that
\begin{equation}
\label{LF_ErsteEigenschaft}
\psi_1(|x|) \leq V(x) \leq \psi_2(|x|), x \in \R^n
\end{equation}
holds and $\exists \chi \in \Kinf$, $\alpha \in \PD$ and continuous function $\varphi:\R_+ \to \R$, $\varphi(0)=0$  such that 
$\forall x \in \R^n$, $\forall u \in \R^m$ it holds
\begin{equation}
\label{ISS_LF_Imp}
V(x)\geq\chi(|u|)\Rightarrow 
\left \{
\begin {array} {l}
{ \dot{V}(x)= \nabla V \cdot f(x,u) \leq - \varphi(V(x))} \\
{ V(g(x,u))\leq  \alpha(V(x)).}
\end {array}
\right.
\end{equation}
If in addition $\varphi(s) = c s$ and $\alpha(s) = e^{-d} s$ for some $c,d \in \R$, then $V$ is called \textit{exponential ISS-Lyapunov function with rate coefficients} $c,d$.
\end{Def}
Note that our definition of ISS-Lyapunov function is given in an implication form. The dissipation form is used, e.g., in \cite{HLT08}.

%

We provide Lyapunov type sufficient condition for the system \eqref{ImpSystem} to be ISS. 
The FDT condition \eqref{Gen_Dwell-Time} is taken from \cite{SaP95}, where it was used to guarantee global asymptotic stability of the system without inputs.

Assume, that $x\equiv 0$ is an equilibrium of an unforced system \eqref{ImpSystem}, that is $f(0,0)=g(0,0)=0$.

Define $S_{\theta}:=\{  \{t_i\}_1^{\infty} \subset [t_0,\infty) \ : \ t_{i+1}-t_i \geq \theta,\ \forall i \in \N \}$.


\begin{Satz}
\label{NonLin_DwellTime_Satz}
Let $V$ be an ISS-Lyapunov function for \eqref{ImpSystem} and $\varphi,\alpha$ are as in the Definition \ref{ISS_Imp_LF} and $\varphi \in \PD$. Let for some $\theta, \delta >0$ and all $a>0$ it holds 
\begin{equation}
\label{Gen_Dwell-Time}
\int_a^{\alpha(a)} {\frac{ds}{\varphi(s)}} \leq \theta - \delta.
\end{equation}
Then \eqref{ImpSystem} is ISS for all impulse time sequences $T \in S_{\theta}$. 
\end{Satz}
\begin{proof}
Fix an arbitrary admissible input $u$, $\phi_0 \in \R^n$
and choose the sequence of impulse times $T=\{t_i\}_{i=1}^{\infty}$, $T \in S_{\theta}$. 

Define a function $\tilde{\alpha}: \R_+ \to \R_+$ by
\[
\tilde{\alpha}(x):= \max\{ \max_{ 0 \leq s \leq \chi(x) } \alpha(s), \chi(x)\},\quad x \in \R_+.
\]
Let us introduce the sets
\[
I_1:=\{x \in \R^n: V(x) \leq \chi(\|u\|_{\infty})\},
\]
\[
I_2:=\{x \in \R^n: V(x) \leq \tilde{\alpha}(\|u\|_{\infty})\} \supseteq I_1.
\]
For the sake of brevity we denote 
$y(\cdot):=V(x(\cdot))$.

At first take arbitrary $t_i,\ t_{i+1} \in T$ and assume that $x(t) \notin I_1$ for all $t \in [t_i,t_{i+1}]$. We are going to find an estimate of $y$ on $[t_i,t_{i+1}]$.

From $x(t) \notin I_1$ for all $t \in [t_i,t_{i+1}]$ it follows, that
$y(t) \geq \chi(\|u\|_{\infty}) \text{ for } t \in [t_{i},t_{i+1}]$,
which by \eqref{ISS_LF_Imp} implies 
\[
 \dot{y}(t) \leq -\varphi (y(t)), t \in (t_{i},t_{i+1}).
\]
Integrating, we obtain
\begin{equation}
\label{HilfUngl_1}
\int_{t_i}^{t} \frac{dy(\tau)}{\varphi(y(\tau))} \leq - (t-t_{i}), \quad t \in (t_{i},t_{i+1}).
\end{equation}
Fix any $r>0$ and define $F(q)=\int_{r}^{q} \frac{ds}{\varphi(s)}.$
Note that $F:(0,\infty) \to \R$ is a continuous increasing function. Thus, it is invertible on $(0,\infty)$ and $F^{-1}:\R \to (0,\infty)$ is also an increasing function.

Changing variables in \eqref{HilfUngl_1}, we can rewrite \eqref{HilfUngl_1} as
\begin{equation}
\label{HilfUngl_2}
F(y(t)) - F(y(t_i) ) \leq - (t-t_{i}).
\end{equation}
Consequently, for $t \in [t_i,t_{i+1})$ it holds
\begin{equation}
\label{V_Abschaetzung}
y(t) \leq F^{-1}\left( F(y(t_i)) - (t-t_{i}) \right).
\end{equation}
Taking in \eqref{HilfUngl_2} limit when $t \to t_{i+1}$ and recalling that $t_{i+1}-t_i \geq \theta$, we obtain
\begin{equation}
\label{HilfUngl_3}
F(y^-(t_{i+1})) - F(y(t_i)) \leq - \theta.
\end{equation}
Now we estimate
\[
F(y(t_{i+1})) - F(y(t_i)) \leq \left( F(\alpha(y^-(t_{i+1})))- F(y^-(t_{i+1})) \right)
+ \left( F(y^-(t_{i+1})) - F(y(t_i)) \right).
\]
From \eqref{Gen_Dwell-Time} and \eqref{HilfUngl_3} we obtain 
\[
F(y(t_{i+1})) - F(y(t_i)) \leq (\theta - \delta) - \theta = - \delta.
\]
From this inequality we have
\begin{equation}
\label{Lyap_Abschaetzung}
y(t_{i+1}) \leq F^{-1} ( F(y(t_i))- \delta ).
\end{equation}
In particular, $y(t) < y(t_i)$, $t \in (t_i,t_{i+1}]$.

Define function $\tilde{\beta}$ by:
\[
\tilde{\beta} (r,0) = \max\{r,\alpha(r)\}, \quad \tilde{\beta} (r,t_{i}-t_0) = \underbrace{\zeta \circ \ldots \circ \zeta}_{i \mbox{ times}}  (\tilde{\beta} (r,0)),
\]
where $\zeta(r)=F^{-1}( F(r) -\delta)$.

For $t \in [t_i,t_{i+1})$ define $\tilde{\beta}(r,\cdot-t_0)$ as a continuous decreasing function, such that 
\[
\tilde{\beta} (r,t-t_0) \geq F^{-1}( F(y(t)) - (t-t_{i})),
\]
where $y(\cdot)$ is a trajectory, corresponding to $y(t_i)=\tilde{\beta}(r, t_i-t_0)$.
By construction, for all $t \geq t_0$ it holds
\[
y(t) \leq \tilde{\beta} (y_0,t-t_{0}).
\]
and $\tilde{\beta}$ is continuous, increasing w.r.t. the first argument and decreasing w.r.t. the second. We are going to prove, that for all $r \geq 0$ $\tilde{\beta}(r,t) \to 0$ as soon as $t \to \infty$.
To prove this it is enough to prove, that $z(t_i) := \tilde{\beta}(r,t_i-t_0) \to 0$, when $i \to \infty$.

Let it be false, then for some $r>0$ $\exists \lim\limits_{i \to \infty} z(t_i)= b \neq 0$.
Define $c=\min_{b \leq s \leq z(0)} \varphi(s)$ and observe by the middle-value theorem that
\begin{eqnarray*}
\delta  \leq  F(z(t_{i})) - F(z(t_{i+1})) = \int_{z(t_{i+1})}^{z(t_{i})} \frac{ds}{\varphi(s)}  \leq  \frac{1}{c} (z(t_{i}) - z(t_{i+1})).
\end{eqnarray*}
Hence $z(t_{i}) - z(t_{i+1}) \geq c \delta,$
and the sequence $z(t_{i})$ does not converge, which leads to a contradiction.

Function $\tilde{\beta}$ satisfies all properties of $\KL$-functions, apart from continuity w.r.t. the first argument. But then we can always majorize it by another $\KL$-function, therefore to keep the notation simpler we assume, that $\tilde{\beta} \in \KL$. Hence
\[
V(x(t)) \leq \tilde{\beta}(V(\phi_0),t-t_0), \quad \forall t:\ x(t) \notin I_1.
\]
Let $t^*:=\min \{t: x(t) \in I_1\}$. 
From \eqref{LF_ErsteEigenschaft} we obtain
\begin{equation}
\label{Beta_Abschaetzung}
|x(t)| \leq \beta(|\phi_0|,t-t_0),\quad t \leq t^*,
\end{equation}
where $\beta(r,t)=\psi^{-1}_1( \tilde{\beta}(\psi_2(r),t))$.

 Now let us prove, that $x(t) \in I_2$ for all $t \geq t^*$.

At first note that a trajectory can leave $I_1$ only by a jump. 
If $\|u\|_{\infty}=0$, then $I_1$ is invariant under continuous dynamics, because $x \equiv 0$ is an equilibrium of \eqref{ImpSystem}. 
Let $\|u\|_{\infty}~>~0$ and let for some $t>t^*$ we have $x(t) \in \partial I_1$, i.e. $y(t)=\chi(\|u\|_{\infty})$. 
Then according to the inequality \eqref{ISS_LF_Imp} it holds $\dot{y}(t) \leq -\varphi(y(t)) < 0$ and thus $y(\cdot)$ cannot leave $I_1$ at time $t$.

Now let for some $t_k \in T$, $t_k \geq t^*$ it holds $x(t_k) \notin I_1$ and for some $\eps >0$ $x(t) \in I_1$ for all $t \in (t_k-\eps,t_k)$. 
Then $x(t_k) \in I_2$ by construction of the set $I_2$.

But we have proved, that $y(t) < y(t_k)$ as long as $t>t_k$ and $x(t) \notin I_1$. Consequently, $x(t) \in I_2$ for all $t > t^*$. 

Thus, for $t> t^*$ it holds $V(x(t)) \leq \tilde{\alpha}(\|u\|_{\infty})$
or equivalently
\[
|x(t)| \leq \psi^{-1} (\tilde{\alpha}(\|u\|_{\infty})):= \tilde{\gamma} (\|u\|_{\infty}).
\]
Function $\tilde{\gamma}$ is positive definite and nondecreasing, thus, it may be always majorized by the $\K$-function $\gamma$. 
Recalling \eqref{Beta_Abschaetzung} we obtain that \eqref{ISS_ImpSys} holds
$\forall t \geq t_0$.
\end{proof}

\begin{remark}
Note, that if the discrete dynamics does not destabilize the system, i.e. $\alpha(a) \leq a$ for all $a \neq 0$, then the integral in the right hand side of \eqref{Gen_Dwell-Time} is non-positive for all $a \neq 0$, and so the dwell-time condition \eqref{Gen_Dwell-Time} always holds, which means, that the system is ISS for all impulse time sequences.
\end{remark}

\subsection{Example}
We illustrate the application of our theorem on an academical example.
Let $T$ be an impulse time sequence. Consider the system $\Sigma$, defined by
\begin{eqnarray}
\label{Beisp_NichtLin}
\left\{
\begin{array}{rl}
\dot{x} =& -x^3 + u ,\ t \notin T \\
x(t) = & x^-(t)+ (x^-(t))^3 + u^-(t),\ t \in T.
\end{array}
\right.
\end{eqnarray}
Consider a function $V:\R \to \R^+$, defined by $V(x)=|x|$.
We are going to prove, that $V$ is an ISS Lyapunov function of the system \eqref{Beisp_NichtLin}.\\
The Lyapunov gain $\chi$ we choose by $\chi(r)=\left( \frac{r}{a}\right)^{\tfrac{1}{3}}$, $r \in \R^+$, for some $a \in (0,1)$. \\
Condition $|x| \geq \chi(|u|)$ implies
\begin{eqnarray*}
\dot{V}(x) & \leq& -(1-a) (V(x))^3, \\
V(g(x,u)) &\leq & V(x)+(1+a)(V(x))^3.
\end{eqnarray*}
Let us compute the integral in the left hand side of \eqref{Gen_Dwell-Time}:
\begin{eqnarray*}
I(y,a) = \int_y^{y+(1+a)y^3} \frac{dx}{(1-a)x^3} = \frac{1+a}{2(1-a)} \frac{2+(1+a)y^2}{(1+(1+a)y^2)^2} \leq \frac{1+a}{(1-a)}.			
\end{eqnarray*}
For every $\eps>0$ there exist $a_{\eps}$ such that $I(y,a) \leq 1 + 2\eps$. 

Thus, for arbitrary $\eps>0$ we can choose $\theta:=1+\eps$.
Note, that the smaller $\theta$ we take, the larger are the gains. This demonstrates the trade-off between the size of gains and the density of allowable impulse times.

\subsection{Sufficient condition for exponential ISS-Lyapunov functions}

Theorem \ref{NonLin_DwellTime_Satz} can be used, in particular, for the systems, which possess exponential ISS-Lyapunov functions, but for this particular class of systems stronger result can be proved. 

For a given sequence of impulse times denote by $N(t,s)$ the number of jumps within time-span $(s,t]$.
\begin{Satz}
\label{ExpCase}
Let $V$ be an exponential ISS-Lyapunov function for \eqref{ImpSystem} with corresponding coefficients $c \in \R$, $d \neq 0$. For arbitrary function $h:\R_+ \to(0,\infty)$, for which there exist $g \in \LL$: $h(x) \leq g(x)$ for all $x \in \R_+$
consider the class $\SSet[h]$ of impulse time-sequences, satisfying the generalized average dwell-time (gADT) condition:
\begin{equation}
\index{dwell-time condition!generalized average}
\index{generalized ADT}
\label{Dwell-Time-Cond}
-dN(t,s) - c(t-s) \leq  \ln h(t-s), \quad \forall t\geq s \geq t_0.
\end{equation}
Then the system \eqref{ImpSystem} is uniformly ISS over $\SSet[h]$.
\end{Satz}


\begin{proof}
Pick any $h$ as in the statement of the theorem.
Fix arbitrary input $u$, state $\phi_0$, initial time $t_0$ and choose the increasing sequence of impulse times $T=\{t_i\}_{i=1}^{\infty} \in \SSet[h]$. 


Due to the right-continuity of $x(\cdot)$ the time-span $[t_0,\infty)$ can be decomposed into subspans as 
$[t_0,\infty) = \cup_{i=0}^{\infty} [t^*_i,t^*_{i+1})$,
so that $\forall k \in \N \cup \{0\}$ the following inequalities hold
\begin{equation}
\label{GegebeneSchranke1}
V(x(t)) \geq \chi(\|u\|_{\infty}) \text{ for } t \in [t^*_{2k},t^*_{2k+1}),
\end{equation}
\begin{equation}
\label{GegebeneSchranke2}
V(x(t)) < \chi(\|u\|_{\infty})  \text{ for } t \in [t^*_{2k+1},t^*_{2k+2}).
\end{equation}

Let us estimate $V(x(t))$ on the time-interval $I_k=(t^*_{2k},t^*_{2k+1}]$ for arbitrary $k \in \N \cup \{0\}$.
Within timespan $I_k$ there are $r_k:=N(t^*_{2k},t^*_{2k+1})$ jumps at times $t^{k}_1,\ldots, t^{k}_{r_k}$. 
To simplify notation, we denote also $t^k_0:=t^*_{2k}$.

For $t \in (t^k_i,t^k_{i+1}]$, $i=0,\ldots,r_k$ we have $V(x(t)) \geq \chi(\|u\|_{\infty})$, thus from \eqref{ISS_LF_Imp} we obtain 
\[
 \dot{V}(x(t)) \leq -c V(x(t)),
\]
and so 
\[
V(x^-(t^k_{i+1})) \leq e^{-c(t^k_{i+1}-t^k_{i})} V(x(t^k_{i})).
\]

At the impulse time $t=t^k_{i+1}$ we know from \eqref{ISS_LF_Imp} that 
\[
 V(x(t^k_{i+1})) \leq e^{-d} V(x^-(t^k_{i+1}))
\]
and consequently 
\[
V(x(t^k_{i+1})) \leq e^{-d-c(t^k_{i+1}-t^k_{i})} V(x(t^k_{i})).
\]

For all $t \in I_k$ we have the following inequality
\[
V(x(t)) \leq e^{-d\cdot N(t,t^*_{2k}) -c(t-t^*_{2k})} V(x(t^*_{2k})).
\]
Dwell-time condition \eqref{Dwell-Time-Cond} implies
\begin{eqnarray}
\label{IntervallSchranke}
V(x(t)) \leq h(t-t^*_{2k}) V(x(t^*_{2k})).
\end{eqnarray}

Take $t^*:=\inf\{t \geq t_0: V(x(t)) \leq \chi(\|u\|_{\infty})\}$. Let us bound the trajectory on $[t_0,t^*]$ by $\KL$-function.

According to \eqref{IntervallSchranke} on $[t_0,t^*]$ it holds
\begin{eqnarray}
\label{KL_Schranke}
V(x(t)) \leq  h(t-t_0) V(x(\phi_0)).
\end{eqnarray}
According to assumptions of the theorem, $\exists g \in \LL$: $h(x) \leq g(x)$ for all $x \in \R_+$.

Using \eqref{LF_ErsteEigenschaft}, we obtain that $\forall t \in [t_0,t^*]$ it holds
\begin{eqnarray*}
|x(t)| \leq \psi_1^{-1}(g(t-t_0) \psi_2(|\phi_0|))=: \beta(|\phi_0|,t-t_0).
\end{eqnarray*}

On the other hand, on the arbitrary interval of the form $[t^*_{2k+1},t^*_{2k+2})$, $k \in \N \cup \{0\}$ we have already the bound on $V(x(t))$ by \eqref{GegebeneSchranke2}.
Since $t^*_{2k+2}$ can be an impulse time, we have an estimate
\[
V(x(t^*_{2k+2})) \leq \max\{1,e^{-d}\} \chi(\|u\|_{\infty}).
\]
From the properties of $h$ it follows, that 
 $\exists C_{\lambda}=\sup_{x\geq 0}\{h(x)\} < \infty$.
Hence for arbitrary $t > t^*$  
we obtain with the help of \eqref{IntervallSchranke} an estimate
\[
V(x(t)) \leq C_{\lambda} \max\{1,e^{-d}\} \chi(\|u\|_{\infty}).
\]

Hence we obtain \eqref{ISS_ImpSys} for all $t \geq t_0$,
where
$\gamma(r)=\psi_1^{-1}(C_{\lambda} \max\{1,e^{-d}\} \chi(r))$.
This proves, that the system \eqref{ImpSystem} is ISS. The uniformity is clear since the functions $\beta$ and $\gamma$ do not depend on the impulse time sequence.
\end{proof}
\begin{remark}
The Theorem \ref{ExpCase} generalizes Theorem 1 from \cite{HLT08}, where this result for constant functions $h \equiv e^{\mu}$ has been proved.
\end{remark}

The condition \eqref{Dwell-Time-Cond} is tight, i.e., if for some sequence $T$ the function $N(\cdot,\cdot)$ does not satisfy the condition \eqref{Dwell-Time-Cond} for every function $h$ from the statement of the Theorem \ref{ExpCase}, then there exists a certain system \eqref{ImpSystem} which is not ISS w.r.t. the impulse time sequence $T$.
This one can see from the simple example. Consider
\begin{equation*}
\left \{
\begin {array} {l}
\dot{x}=-cx,\\
x(t)=e^{-d}x^-(t)
\end {array}
\right.
\end {equation*}
with initial condition $x(0)=x_0$.
Its solution for arbitrary time sequence $T$ is given by
\[
x(t)=e^{-dN(t,t_0)-c(t-t_0)}x_0.
\]
If $T$ does not satisfy the gADT condition, then $e^{-dN(t,t_0)-c(t-t_0)}$ cannot be estimated from above by $\LL$-function, and consequently, the system under consideration is not GAS.

\section{Conclusion}
\label{Schluss}

In this paper we have proved two Lyapunov-type sufficient conditions, which ensure ISS of an impulsive system. The first theorem states that existence of an ISS-Lyapunov function implies ISS of an impulsive system for the impulse time sequences which satisfy the nonlinear fixed dwell-time condition. The other one states, that 
existence of an \textit{exponential} ISS-Lyapunov function implies \textit{uniform} ISS of an impulsive system for the impulse time sequences which satisfy a generalized average dwell-time condition.

One of the possible directions for a future research is to prove Lyapunov small-gain theorems for ISS of interconnections of impulsive systems, which subsystems possess either exponential or non-exponential ISS-Lyapunov functions.

\bibliographystyle{plain}


\bibliography{literatur}

\end{document}